\documentclass[11pt openany]{article}

\usepackage{graphics,psfrag,graphicx,subfigure,epsfig}
\usepackage{amsmath,amsfonts,amssymb,amstext,amsthm,amscd,mathrsfs}
\usepackage[spanish,english]{babel}
\usepackage[all]{xy}
\usepackage{multicol}
\usepackage{rotating}

\theoremstyle{definition}
\newtheorem{defi}{Definition}
\newtheorem{obs}{Observation}

\newtheorem{rmk}{Remark}

\theoremstyle{plain}
\newtheorem{theo}{Theorem}
\newtheorem{coro}{Corollary}
\newtheorem{prop}{Proposition}
\newtheorem{lema}{Lemma}
\newtheorem{con}{Conjecture}

\newcommand{\noi}{{\noindent}}
\newcommand{\su} {\mathbf{U}}
\newcommand{\gl} {\mathbf{GL}}
\newcommand{\Z}{\mathbb{Z}}
\newcommand{\C}{\mathbb{C}}
\newcommand{\R}{\mathbb{R}}
\newcommand{\CP}{\mathbb{CP}}
\newcommand{\T} {\mathbb{T}}
\newcommand{\vdu}{\mathcal{N}}
\newcommand{\vd} {\mathcal {M}}
\newcommand{\lv} {\mathbf {LV}}
\newcommand{\mer}{\mathbf {M}}
\newcommand{\ps}{\mathbf {PS}}

\newcommand{\Lam}{\pmb{\Lambda}}
\newcommand{\ZLam}{Z^{^\C}(\Lam)}
\newcommand{\ZLamsu}{{\pmb{Z}}^{^\C}_{(\Lam,n,1)}}
\newcommand{\ZLamss}{{\pmb{Z}}^{^\C}_{(\Lam,n,s)}}

\newcommand{\ZLamp}{Z^{^\C}_{_0}(\Lam)}
\newcommand{\ZLamd}{Z^{^\C}_{_+}(\Lam)}

\newcommand{\ZLamu}{Z^{^\C}(\Lam')}
\newcommand{\s}{\mathbb{S}}
\newcommand{\disc}{\mathbb{D}}
\newcommand{\tange} {\mathbf{T}}
\newcommand{\p} {\mathcal{P}}
\newcommand{\K} {\mathcal{K}}
\newcommand{\vt} {\mathbf{N}}
\newcommand{\fol} {\mathcal{F}}
\newcommand{\te}{\tilde{\mathcal{E}}}
\newcommand{\cv}{\mathcal{X}}
\newcommand{\U} {\mathcal{U}}

\newcommand{\V} {\mathcal{V}}
\newcommand{\ls}{\mathcal{L}}

\begin{document}
\title{Open Book Structures on Moment-Angle Manifolds $\ZLam$ and Higher Dimensional Contact Manifolds.\thanks{Partially supported by project PAPIIT-DGAPA IN100811 and IN108112 and project CONACyT 129280.}}
\author{Yadira Barreto, Santiago L\'opez de Medrano and \\ Alberto Verjovsky }
\date{}

\maketitle

\begin{abstract}
\noi We construct  open book structures on moment-angle ma\-ni\-folds\- and give a new construction of examples of contact manifolds in arbitrarily large dimensions.
\end{abstract}

\noi {\bf Key Words:} Open book decomposition, moment-angle manifolds,  contact structures.

\section{Introduction}
The topology of generic intersections of quadrics in $\R^n$ of the form:
$$
\sum_{i=1}^n\lambda_ix^2_i=0,\quad\quad\sum_{i=1}^n x^2_i=1,
$$
where $\lambda_i\in\R^k$, $i=1,\dots,n$ has been studied for many years: For $k=2$ they are  diffeomorphic to a triple product of spheres or to the connected sum of sphere products (\cite{SLM}, \cite {GL}); for $k>2$ this is no longer the case (\cite{BBCG}, \cite{BM}) but there are large families of them which are again connected sums of spheres products (\cite{LoGli}).\\

\noi  Let $\Lam=\left(\lambda_1,\cdots, \lambda_n  \right)$.
The generic condition, known as \emph{weak hyperbolicity} and equivalent to regularity, is the following:\\

\noi\emph{If $J\subset{1,\dots,m}$ has $k$ or fewer elements then the origin is not in the convex hull of the $\lambda_i$ with $i\in J$.}\\

\noi A crucial aspect of these varieties is that they admit natural group actions. All of them admit $\Z^{^n}_2$ actions. Their complex versions in $\C^n$

$$
\sum_{i=1}^n\lambda_i|z_i|^2=0,\quad\quad \sum_{i=1}^n|z_i|^2=1,
$$
\noi (now known as \emph{moment angle manifolds}) admit natural $\T^n$ actions. The quotient is in both cases a polytope $\p$ that determines completely the ma\-ni\-folds\- and the actions. We will use both notations  $Z=Z(\Lam)=Z(\p)$ and $Z^{^\C}=\ZLam=Z^{^\C}(\p)$ for these manifolds\footnote{The facets of the polytope correspond with the intersections of the manifolds with the coordinate hyperplanes. Technically, we would have to add the datum $n$ to the notations $Z(\p)$ and $Z^{^\C}(\p)$ to specify that some \textit{facets} may be empty. In Theorem 1 this would mean that the binding might be empty. In all our examples this will be explicit.}. They both fall under the general concept of \emph{generalized moment angle complexes} (\cite{BP} and \cite{BBCG}).\\

\noi An open book construction with these manifolds was used to provide a des\-crip\-tion of the topology of some cases not covered by the main theorem in \cite{SLM} (see remark on page 281). In \cite{LoGli} it is a principal technique for proving the results about some families when $k>2$.\\

\noi In section 2 we recall this construction, underlining its consequences for moment-angle manifolds:\\

\noi If $\p$ is a simple convex polytope and $F$ is one of its facets, there is an open book decomposition of $Z^{\C}(\p)$ with binding $Z^{^\C}(F)$ and  trivial mo\-no\-dro\-my.\\

\noi When $k=2$, we also give a topological description of the binding and leaves of the decomposition, in terms of the \emph{odd cyclic partitions} of $n$ that classify the varieties\footnote{See section 7.}. This description is almost complete in the case of moment-angle manifolds for $k=2$: The leaf is the interior of a manifold that can be:\\

\begin{itemize}
\item [a)] a product $\s^{2n_2-1}\times\s^{2n_3-1}\times\disc^{2n_1-2}$,

\item [b)] a connected sum along the boundary of products  of the form 
$\s^p\times\disc^{2n-p-4}$,

\item [c)] in some cases, in the connected sum there may appear summands of the form $\s^{2p-1}\times\s^{2n-2p-3}\backslash\disc^{2n-4}$ or of 
the exterior of an embedded $\s^{2q-1}\times\s^{2r-1}$ in $\s^{2n-4}$.
\end{itemize}

\noi The specific products that appear in the above formulae will be described precisely in \ref{pageC} in section 3. The proofs will be postponed to section 7 where they will follow from a more general theorem for the corresponding real manifolds that requires, as in \cite{SLM}, additional dimensional and connectivity hypotheses. The  corresponding result in the real case amounts to the topological description of the \emph{half} manifolds $Z_{_+}=Z\cap\{x_1\geqslant 0\}$, complementing the results in \cite{SLM}.  Parts of these theorems follow directly from the results in \cite{SLM}, other parts require the use of many of the technical lemmas proved there. All these manifolds with boundary are also generalized moment-angle complexes.\\

\noi In section 4 we recall some results in \cite{MeerVerj1} which state that the $\lv$-$\mer$ complex manifolds $\vdu(\Lam)$ which correspond to configurations $\Lam$ which satisfy an arithmetic condition (Condition ({\bf K}))  fibre (possibly as Seifert fibrations) over toric varieties (or orbifolds with simple singularities) with fibre a complex torus. From this we conclude that the pages of the open books described in section 2 are complex manifolds of the form $\vdu(\Lam)-\vdu(\Lam_0)$ where
$\Lam_0$ is obtained from $\Lam$ by suppressing one element of the configuration.\\

\noi In section 5 we show that every moment-angle manifold $\ZLam$ admits an almost contact structure and, as a consequence, these manifolds also admit a quasi-contact structure. We recall that there is the conjecture \cite{MMP} that every almost contact manifold is in fact a contact manifold.\\ 

\noi In section 6 we describe a new construction of compact contact manifolds $\ZLamss$ in arbi\-tra\-ri\-ly large odd dimensions. 
These manifolds are generalizations of the moment-angle manifolds $\ZLam$ that have been studied by V. G\'omez Guti\'errez and the second author \cite{GL}.\\

\noi Even dimensional moment-angle manifolds can be given complex structures but admit symplectic structures only in a few well-known cases \cite{Meer} so it is surprising that many of the odd dimensional ones admit contact structures. In this respect we remark that using results of C. Meckert \cite{Meckert} and Y. Eliashberg \cite{Elias} it is possible to give a contact structure on these manifolds.  See remark \ref{rmk4} in section 6.\\

\noi However our method is distinct and in some sense explicit without using open books. What we use is the heat flow method described in \cite{AltWu}. In other words, we construct a positive confoliation $\alpha$ which is conductive i.e. every point of the manifold can be connected by a suitable Legendrian curve to a point where the form is positive (see definition 4 in 6.3 in section 6).\\

\noi In section 7 we recall some old results about the topology of intersections of quadrics and use them to proof a new one from which theorem \ref{pageC} of section 2 follows.

\section{Open book decomposition}

Let us recall the concept of an \emph{open book}, introduced by H. E. Winkelnkemper in 1973 (see \cite{Win}). An open book in a smooth manifold $V$ is a pair $(\K,f)$ consisting of the following:

\begin{itemize}
\item [a)]  A proper submanifold $\K$  of codimension two in $V$ with trivial normal bundle, so $\K$  admits  a neighborhood $\vt$ diffeomorphic to $\disc^2\times\K$.

\item [b)] A locally trivial smooth fibration $f:V\backslash\K\to \s^1$ such that there exists a neighborhood $\vt$ of $\K$ as above, in which $f$ is the normal angular coordinate.
\end{itemize}
\noi The submanifold $\K$ is called the \emph{binding} and the 
fibres of $f$ are the \emph{pages} of the open book $(\K,f)$.\\

\begin{center}
\begin{figure}[h]
\centerline{\includegraphics[height=4cm]{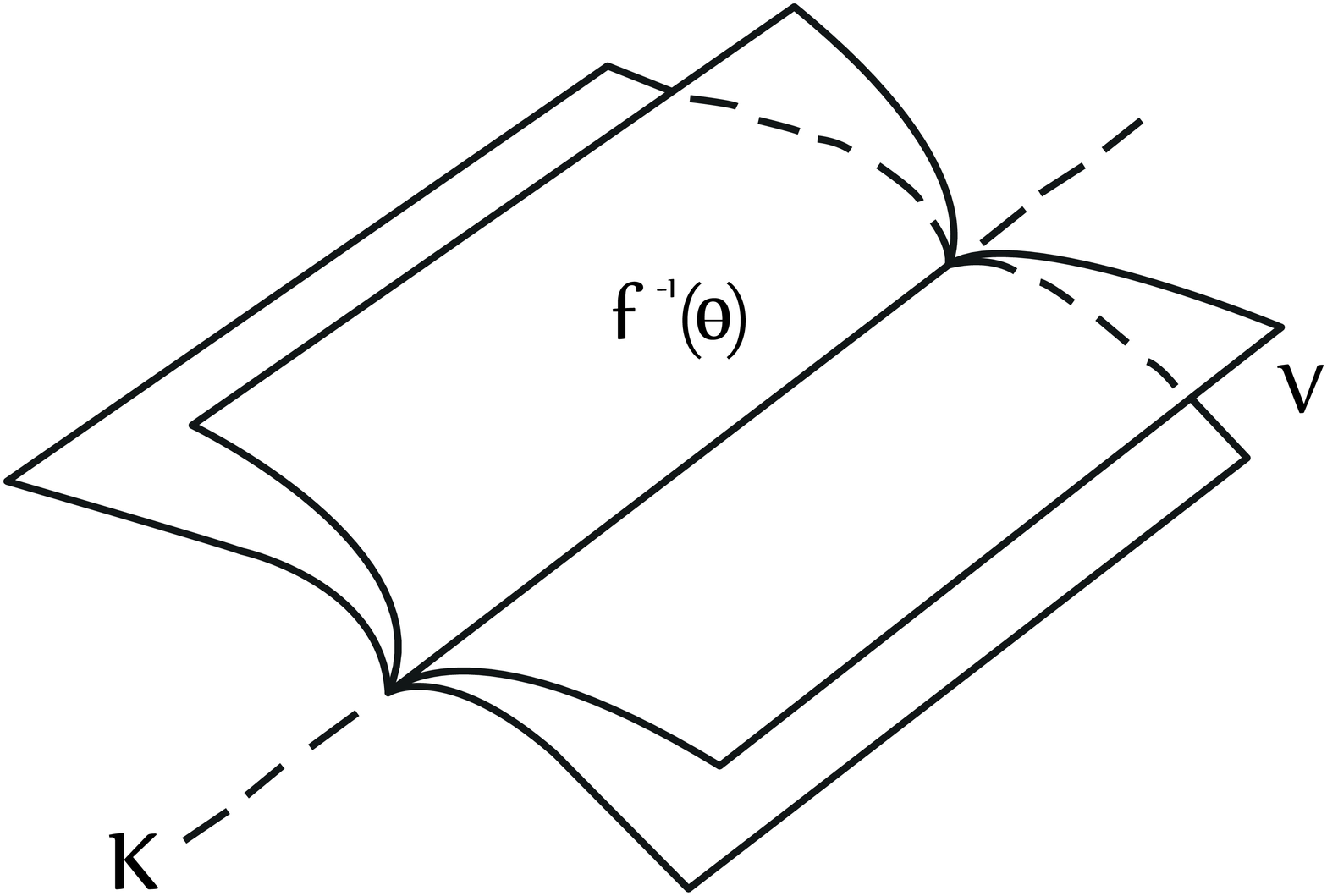}}
\end{figure}
\end{center}   

\noi Also we can define an open book as follows: Let $F$ be a manifold with a boun\-da\-ry and $\phi$ a self diffeomorphism of $F$ that is the identity on $\partial F$. The mapping torus $\Sigma(F,\phi)$ of $\phi$, i.e., the quotient of $F\times\R$ by the equivalence relation ge\-ne\-ra\-ted by $(p,t)\sim\left(\phi(p),t+1\right) $, is a manifold with boun\-da\-ry $\partial\Sigma(F,\phi)=\partial F\times\R/\Z$. Collapsing each circle $\{p\}\times\R/\Z$, to a point, $\{p\}\in\partial F$ we obtain a manifold $\bar{\Sigma}(F,\phi)$ without boundary called the {\it relative mapping torus} of $\phi$.\\

\noi This manifold has an obvious open book structure whose binding is a copy of $\partial F$, the collapsed $\partial F\times \R/\Z$, and whose mapping to the circle is induced by the projection $F\times\R\to\R$ and so the pages are copies of $F$.\\

\noi We call \emph{monodromy} of an open book $(\K,f)$ in a manifold $V$ to any self diffeomorphism of a page $F$,  such that there is a diffeomorphism $\bar{\Sigma}(F,\phi)\to V$ taking the obvious open book to $(\K,f)$.\\

\noi Passing to our varieties, $Z$ admits a $\Z^{^n}_2$ action obtained by changing the signs of the coordinates. The quotient is a simple polytope $\p$ which can be identified with the intersection of $Z$ and the first orthant of $\R^n$. It follows that $Z$ can be re\-cons\-truc\-ted from this intersection by reflecting it on all the coordinate hyperplanes.\\

\noi By a simple change of coordinates $r_i=x^2_i$, this quotient can be identified with the $d$-dimensional convex polytope given by

$$
\sum_{i=1}^n\lambda_ir_i=0,\quad\quad\sum_{i=1}^nr_i=1,
$$
$$
r_i\geqslant 0.
$$

\noi Let $\Lam'$ be obtained from $\Lam$ by adding an extra $\lambda_1$ which we interpret as the coefficient of a new extra variable $x_0$, so we get the manifold $Z'$:
$$
\lambda_1\left(x^2_0+x^2_1\right)+\sum_{i>1}\lambda_ix^2_i=0,
$$
$$
x^2_0+x^2_1+\sum_{i>1}x^2_i=1.
$$

\noi Let $Z_+$ be the intersection of $Z$ with the half space $x_1\geqslant 0$ and $Z_{_+}'$ the intersection of $Z'$ with the half space $x_0\geqslant 0$.\\

\noi The boundary of $Z_{_+}'$ is $Z$ (this shows that $Z$ is always the boundary of a pa\-ra\-lle\-li\-za\-ble manifold).\\

\noi $Z_+$ admits an action  of $\Z^{^{n-1}}_2$ by changing signs on all the variables except $x_1$ and the quotient is again $\p$. In other words, $Z_+$ can be obtained  from $\p\times \Z^{^{n-1}}_2$ by reflecting $\p$ on all the coordinate hyperplanes except $x_1=0$.\\

\noi Consider also the manifold $Z_0$ which is the intersection of $Z$ with the subspace $x_1=0$. $Z_0$ is the boundary of $Z_+$.\\

\noi $\s^1$ acts on $Z'$ (rotating the coordinates $\left(x_0,x_1\right)$) with fixed set $Z_0$. Its quotient can be identified with $Z_+$. The map

$$
\left(x_0, x_1, x_2,\dots, x_m\right) \mapsto \left(\sqrt {x^2_{_0}+x^2_{_1}}, x_2,\dots, x_m\right)
$$
 is a retraction from $Z'$ to $Z_+$ which restricts to the retraction from $Z$ to $Z_+$
 $$
 \left(x_1,x_2,\dots,x_m\right) \mapsto\left(|x_1|,x_2,\dots,x_m\right).$$

\noi Observe further that this retraction restricted to $Z'_{_+}$ is homotopic to the identity: the homotopy preserves the coordinates $x_i$, $i\geqslant 2$ and folds gradually the half space $x_0 \geqslant 0$ of the $x_0$, $x_1$ plane into the ray $x_0=0$, $x_1\geqslant 0$ preserving the distance to the origin.\\

\noi So $Z$ is the double of $Z_+$ and $Z'$ is the double of $Z'_{_+}$, and $Z'$ is the open book with binding $Z_0$, page $Z_+$ and trivial monodromy.

\begin{theo}\label{bookR}
Every manifold $Z'$ is an open book with trivial monodromy whose binding is $Z_0$ and page $Z_+$.
\end{theo}

\noi Since the manifold $Z^{^\C}(\p)$ can be considered for each $i$ as a manifold $Z'$ with repeated coefficient $\lambda_i$, then it is an open book with binding the manifold obtained by taking $z_i=0$. This proves:

\begin{theo}\label{bookC}
If $\p$ is a simple convex polytope and $F$ is one of its facets, there is an open book decomposition of $Z^{^\C}(\p)$ with binding $Z^{^\C}(F)$ and  trivial monodromy.
\end{theo}

\noi When $k=2$ the topology of $Z^{^\C}$ can be described precisely (see \cite{SLM} and section 8) and that includes also that of all the bindings. It can be expressed in terms of the \emph{cyclic partition} associated to $Z^{^\C}$. We have a precise description of the topology of the leaves in most cases. For our purposes we will only need the case where the total manifold is a moment-angle manifold which is described in the following:

\begin{theo}\label{pageC}
Let $k=2$, and consider the manifold $Z^{^\C}$ corresponding to the cyclic partition $n=n_1+\dots+n_{2\ell+1}$. Consider the open book decomposition of $Z^{^\C}$ corresponding to the binding at $z_1=0$, as given by Theorem \ref{bookC}. Then the leaf of this decomposition is diffeomorphic to the interior of:

\begin{itemize}
\item [a)] If $\ell=1$, the product
$$
\s^{2n_2-1}\times\s^{2n_3-1}\times\disc^{2n_1-2}.
$$

\item [b)] If $\ell>1$ and $n_1>1$, the connected sum along the boundary of $2\ell+1$ manifolds:
$$
\coprod_{i=2}^{\ell+2}\left(\s^{2d_i-1}\times\disc^{2n-2d_i-3}\right)\coprod\coprod_{i=\ell+3}^1\left(\disc^{2d_i-2}\times\s^{2n-2d_i-2}\right).
$$

\item [c)] If $n_1=1$ and $\ell>2$, the connected sum along the boundary of $2\ell$ manifolds:
$$
\coprod_{i=3}^{\ell+1}\left(\s^{2d_i-1}\times\disc^{2n-2d_i-3}\right)\coprod\coprod_{i=\ell+3}^1\left(\disc^{2d_i-2}\times\s^{2n-2d_i-2}\right)
$$
$$
\coprod\left(\s^{2d_2-1}\times\s^{2d_{\ell+2}-1}\backslash\disc^{2n-4}\right).
$$

\item [d)] If $n=1$ and $\ell=2$, a connected sum along the boundary of two manifolds:
$$
\left(\s^{2d_2-1}\times\s^{2d_4-1}\backslash \disc^{2n-4}\right)\coprod\te,
$$
where $\te$ is the exterior of an embedded $\s^{2n_2-1}\times\s^{2n_5-1}$ in $\s^{2n-4}$.
\end{itemize}
\end{theo} 

\noi The proof of this theorem and of its real version will appear in section 7. Similar results can be given when $k>2$ for large families in the spirit of \cite{LoGli}.

\section{Examples of open book decompositions of some moment-angle manifolds.}

\noi To illustrate the variety of decompositions obtained we give now some examples, by direct application of the previous theorem (which gives always the decomposition with binding at $z_1=0$). The reader may get a better feeling if she or he looks at these examples in the light of the proof of the theorem in section 7. An additional feature of these examples is that we obtain three different open book decompositions of the same moment-angle manifold:\\

\noi Let $\Lam=\left(\lambda_1,\lambda_2,\lambda_3,\lambda_4,\lambda_5\right)$. The $5$-tuple in $\C$ corresponding to the five roots of unity ($\Lam$ satisfies the weak hiperbolicity condition).
 
 \begin{enumerate}
\item Consider $\Lam''$ where $\lambda_1$ has multiplicity 3. We obtain the moment-angle manifold 
$$
Z^{^\C}(\Lam)=\underset{2}{\sharp}\;\s^7\times\s^4\;\underset{3}\sharp\;\s^3\times\s^8.
$$

\noi When $z_1=0$ we have a configuration $\Lam'$, where now the multiplicity of $\lambda_1$ is two. Then the binding is 
$$
Z^{^\C}_{_0}(\Lam')=\underset{2}\sharp\;\s^5\times\s^4\;\underset{3}\sharp\;\s^3\times\s^6.
$$

\noi The page, by Theorem \ref{pageC}(b), is in this case the manifold 
$$
Z^{^\C}_{_+}(\Lam)=\left(\coprod_3\;\s^3\times\disc^7\right)\;\coprod\; \left(\coprod_2\;\disc^6\times\s^4\right).
$$

\item Consider  $\Lam''$ where now $\lambda_2$ has multiplicity 3. The moment-angle manifold is the same:
$$
Z^{^\C}(\Lam)=\underset{2}\sharp\;\s^7\times\s^4\;\underset{3}\sharp\;\s^3\times\s^8.
$$ 

\noi When $z_1=0$ the binding is,

 $$Z^{^\C}_{_0}(\Lam')=\s^5\times\s^3\times\s^1.
 $$
\noi and the page,  by Theorem \ref{pageC}(d), is 
$$
Z^{^\C}_{_+}(\Lam)=\s^7\times\s^3\;\big\backslash\;\disc^{10}\coprod \te,$$ 
where $\te$ is the complement of an embedding $\s^5\times\s^1$ in $\s^{10}$.\\

\item Consider now $\Lam''$ where now $\lambda_3$ has multiplicity 3. The moment-angle manifold is again
$$
Z^{^\C}(\Lam)=\underset{3}\sharp\;\s^3\times\s^8\;\underset{2}\sharp\;\s^7\times\s^4.
$$
\noi When $z_1=0$ the binding is
 $$
 Z^{^\C}_{_0}(\Lam')= \s^1\times\s^7\times\s^1.
 $$ 
\noi and the page, by Theorem \ref{pageC}(d), is 
$$
Z^{^\C}_+(\Lam)=\s^7\times\s^3\;\big\backslash\;\disc^{10} \coprod \te,$$
where $\te$ is the complement of the (unique, up to isotopy) embedding of $\s^1\times\s^1$ in $\s^{10}$.
\end{enumerate}

\noi When we take $\lambda_4$ or $\lambda_5$ with multiplicity 3 we obtain, by symmetry the same open book decompositions as in examples $3$ and $2$ above, respectively.

\begin{sidewaystable}
\begin{center}
Let: $\Lam$ be an admissible configuration,\\
$\ZLam$ the moment-angle manifold corresponding to $\Lam$,\\
$\ZLamd$ the page and $\ZLamp$ the binding of the open book decomposition.
\end{center}
\begin{tabular}{|c|c|c|c|} \hline
\parbox{4.5cm}{\bf \vspace{-.1cm} \small Admissible Configuration\vspace{-.1cm}}
&\parbox{5cm}{\bf \vspace{-.1cm}\centering\small Moment-Angle Manifold\vspace{-.1cm}}
&\parbox{4.5cm}{\bf \vspace{-.1cm} \centering\small Binding\vspace{-.1cm}}
&\parbox{4.5cm}{\bf \vspace{-.1cm} \centering \small Page\vspace{-.1cm}}\\\hline

\parbox{4.5cm}
{\vspace{.2cm}\centering$\Lam=(1,1,i,-1-i)$ \vspace{.2cm}}
&\parbox{5cm}
{\vspace{.2cm} \centering
$\ZLam=\T^2\times\s^3$\vspace{.2cm}}
&\parbox{4.5cm}
{\vspace{.2cm}\centering
$\ZLamp=\T^3$\vspace{.2cm} }
&\parbox{4.5cm}
{\vspace{.2cm} \centering
$\ZLamd=\T^2\times\disc^2$\vspace{.2cm} }\\\hline


 \parbox{4.5cm}
{\vspace{.2cm}\centering
$\Lam=(1^{n-2},i,-1-i)$\vspace{.2cm}}
&\parbox{5cm}
{\vspace{.2cm} \centering
$\ZLam=\T^2\times\s^{2n-5}$\vspace{.2cm}}
&\parbox{4.5cm}
{\vspace{.2cm}\centering
$\ZLamp=\T^2\times\s^{2n-7}$\vspace{.2cm} }
&\parbox{4.5cm}
{\vspace{.2cm}\centering
$\ZLamd=\T^2\times\disc^{2n-6}$\vspace{.2cm}}\\\hline


\parbox{4.5cm}
{\vspace{.2cm}\centering$\Lam=(i,i,1,-1-i,-1-i)$ \vspace{.2cm}}
&\parbox{5cm}
{\vspace{.2cm} \centering
$\ZLam=\s^3\times\s^1\times\s^3$\vspace{.2cm}}
&\parbox{4.5cm}
{\vspace{.2cm}\centering
$\ZLamp=\T^2\times\s^3$\vspace{.2cm} }
&\parbox{4.5cm}
{\vspace{.2cm} \centering
$\ZLamd=\s^1\times\s^3\times\disc^2$\vspace{.2cm} }\\\hline


\parbox{4.5cm}
{\vspace{.2cm}\centering\scriptsize
$\Lam=\left(\lambda_1,\lambda_2,\lambda_3,\lambda_4,\lambda_5\right)$\\ where the $\lambda_i$ are the $5-th$ roots of unity.\vspace{.2cm}}
&\parbox{5cm}
{\vspace{.2cm} \centering
$\ZLam=\underset{5}{\sharp}\; \s^3\times\s^4$\vspace{.2cm}}
&\parbox{4.5cm}
{\vspace{.2cm} \centering
$\ZLamp=\T^2\times\s^3$\vspace{.2cm}}
&\parbox{4.5cm}
{\vspace{.2cm}  \centering\scriptsize
$\ZLamd=\left(\s^3\times\s^3\backslash\disc^6\right)\amalg \tilde{\mathcal{E}},
$
where  $\tilde{\mathcal{E}}$ is the exterior of the canonical embedded $\T^2$ in $\s^6$.
\vspace{.2cm}} \\\hline


\parbox{4.5cm}
{\vspace{.2cm}\centering\scriptsize
$\Lam=\left(\lambda_1,\lambda_1,\lambda_2,\lambda_3,\lambda_4,\lambda_5\right)$\\ where the $\lambda_i$ are the $5-th$ roots of unity. \vspace{.2cm}}
&\parbox{5cm}
{\vspace{.2cm} \centering
$\ZLam=\underset{2}\sharp\;\s^5\times\s^4\;\underset{3}\sharp\;\s^\times\s^6$\vspace{.2cm}}
&\parbox{4.5cm}
{\vspace{.2cm}\centering
$\ZLamp=\underset{5}\sharp\;\s^3\times\s^4$\vspace{.2cm} }
&\parbox{4.5cm}
{\vspace{.2cm} \centering\scriptsize 
$\ZLamd=\underset{2}\coprod\left(\disc^4\times\s^4\right)\coprod\underset{3}\coprod\left(\s^3\times\disc^5\right)$\vspace{.2cm}}\\\hline


\parbox{4.5cm}
{\vspace{.2cm}\centering\scriptsize$\Lam=\left(\lambda_1,\lambda_2,\lambda_3,\lambda_4,\lambda_5,\lambda_6,\lambda_7\right)$\\ where $\lambda_i$ are the $7-th$ roots of unity.\vspace{.2cm}}
&\parbox{5cm}
{\vspace{.2cm} \centering
$\ZLam=\underset{7}\sharp\;\s^5\times\s^6$\vspace{.2cm}}
&\parbox{4.5cm}
{\vspace{.2cm}\centering
$\ZLamp=\underset{2}\sharp\;\s^5\times\s^4\;\underset{3}\sharp\;\s^3\times\s^6$\vspace{.2cm}}
&\parbox{4.5cm}
{\vspace{.2cm} \centering\scriptsize 
$\ZLamd=\underset{2}\coprod\left(\s^5\times\disc^5\right)\coprod \left(\underset{3}\coprod\;\disc^4\times\s^6\right)$ \\$\coprod \Big(\left(\s^5\times\s^5\right)\backslash\disc^{10}\Big)$\vspace{.2cm}}\\\hline

\end{tabular}
\caption{\small Open Book Decomposition of some Moment-Angle Manifolds.}
\label{OBD}
\end{sidewaystable}
\newpage

\section{Complex and algebraic structures on moment-angle manifolds,  $\lv$-manifolds and pages.}

If we take the quotient of $\ZLam$ by the scalar (or diagonal) action of $\s^1$: 
$$
\vdu(\Lam)=\ZLam/\s^1,
$$

\noi we obtain a compact, smooth manifold  $\vdu(\Lam)\subset \CP^{n-1}$. These manifolds are called by some authors  $\lv$-$\mer$ manifolds.\\

\noi It is known that various of these objects admit natural complex structures (see, for example, \cite{BM}, Theorem 12.2): When $k$ is even, both $\vdu(\Lam)$ and $\ZLam\times\s^1$ have natural complex structures.  When $k$ is odd, $\ZLam$ itself has a natural complex structure.\\

\noi Let
$$
\pi_{_\Lambda}:\ZLam\to\vdu(\Lam),
$$ 
\noi denote the canonical projection.\\

\noi Consider now the open book decomposition of $\ZLam$ described above, co\-rres\-pon\-ding to the variable $z_1$. Let $\Lam_0$ be obtained from $\Lam$ by removing $\lambda_1$. It is clear that the diagonal $\s^1$-action on $\ZLam$ has the property that each orbit intersects the page $\ZLamd$ in a unique point and every point of this page is intersected tranversally by the orbits. This implies that the restriction of the canonical projection $\pi_{_{\Lambda}}$ to each page is a diffeomorphism onto its image $\vdu(\Lam)-\vdu(\Lam_0)$.\\

\noi For $k$ even we therefore obtain, by pulling-back the complex structure of $\vdu(\Lam)-\vdu(\Lam_0)$:\\

\begin{rmk}\label{rmk1}
\noi \emph{The page of the open book decomposition of $\ZLam$ with binding $Z^{^\C}_{_0}(\Lam_0)$ admits a natural complex structure which makes it biholomorphic to the complex manifold $\vdu(\Lam)-\vdu(\Lam_0)$.}\\
\end{rmk}

\noi We can go a step further by using the results in \cite{MeerVerj1}:\\

\noi It is shown there that for every $\Lam$ the manifold $\vdu(\Lam)$ has a holomorphic and locally-free action of $\C^{k/2}$ whose orbits determine a tranversally K\"ahler foliation $\fol$ of complex dimension $k/2$. The condition for the leaves of $\fol$ to be compact is the following:\\

\noi An admissible configuration $\Lam=\left(\lambda_1,\dots,\lambda_n\right)$ fulfills condition $(\pmb{K})$ if and only if we may choose, for the (real) space of solutions of the system 

\begin{equation*}
\begin{cases}
\sum_{i=1}^n s_i\lambda_i&=0,\\
\sum_{i=1}^ns_i&=0,
\end{cases}
\end{equation*}
\noi a basis with integer coordinates.\\

\noi Also it was shown in \cite{MeerVerj1} that every configuration $\Lam$ can be made to satisfy condition  $(\pmb{K})$ by an arbitrarily small perturbation. The main result proved in \cite{MeerVerj1} is the following:

 \begin{theo}\label{CEF}
 Let $\Lam$ be an admissible configuration that satisfies condition $(\pmb{K})$. Then \\
 
 \noi (a) The leaves of the foliation $\fol$ of $\vdu(\Lam)$ are compact complex tori of complex dimension $m$.\\
 
 \noi (b) The quotient space of $\vdu(\Lam)$ by $\fol$ is a projective toric variety of complex dimension $n-2m-1$. We denote it by $T(\Lam)$.\\
 
  \noi (c) The toric variety $T(\Lam)$ comes equipped with an equivariant orbifold structure.\\
  
\noi (d) The natural projection $p_{_\Lambda} : \vdu(\Lam)\to T(\Lam)$ is a holomorphic principal Seifert bundle, with compact complex tori of complex dimension $m$ as fibers.\\

\noi (e) The transversely K\"ahlerian form $\omega$ of $\vdu(\Lam)$ projects onto a K\"ahlerian form $\tilde{\omega}$ of $T(\Lam)$.  
 \end{theo}
 
 \noi The bundle $p_{_{\Lam}}:\vdu(\Lam)\to T(\Lam)$ is called a  \emph{generalized Calabi-Eckmann fibration} over $T(\Lam)$.  \\

\noi If $\Lam$ satisfies condition $(\pmb{K})$ (and therefore $\Lam_0$ also) the manifolds $T(\Lam)$ and $T(\Lam_0)$ are both toric orbifolds (possibly singular) and therefore both are algebraic varieties. As a consequence, $T(\Lam)-T(\Lam_0)$ is a quasi-projective variety, which (when is nonsingular) is a K\"ahler manifold and, in particular, symplectic.\\

\noi Therefore, we have shown:
\begin{theo}
Assume $k$ is even. Then the leaf of the open book structure on $\ZLam$ is naturally a complex manifold. After a small perturbation of $\Lam$, this leaf admits a holomorphic (Seifert) fibration over a quasitoric variety, with fibre a compact complex torus.
\end{theo} 

\section{Different Geometric Structures.}

\noi The even dimensional moment-angle manifolds and the $\lv$-manifolds have the characteristic that, except for a few, well-determined cases, they do not admit symplectic structures. It was, therefore, surprising to us that the odd dimensional ones can have contact structures. 
In fact we can conjecture that they are all contact manifolds:

\begin{con} Assume $k=2m$ is even. Then, for every admissible configuration $\Lam\subset{\R^{2m}}$, the odd dimensional manifold $\ZLam$ is a contact manifold.

\end{con}

\noi Here is a first example: The next theorem was proved by F. Bourgeois in  \cite{Bour} (it is a corollary of Theorem 10 in \cite{Giroux}).

\begin{theo}\label{Bour}
If a closed manifold $\vd$ admits a contact structure, then so does $\vd\times\T^2$.
\end{theo}

\noi Therefore, \emph{For $n>3$, moment-angle manifolds such as
$$
\ZLam=\s^{2n-5}\times\T^{2m}
$$
admit a contact structure.}\\

\noi We will get closer to our conjecture by showing that odd-dimensional moment-angle manifolds admit structures that are weaker versions of the contact structure.

\subsection{Almost contact manifolds and quasi-contact manifolds.}

\begin{defi}
 A $(2n+1)$-dimensional manifold $\vd$ is called \emph{almost contact} if its tangent bundle admits a reduction to $\su(n)\times \R$.
\end{defi}

\noi It is known that every contact manifold is an almost contact manifold (see \cite{Gray}).\\

\noi On the other hand, D. Mart\'inez, V. Mu\~{n}oz and F. Presas in \cite{MMP} defined a quasi-contact manifold as follows.
\begin{defi}
A $(2n+1)$-dimensional manifold $\vd$ is called \emph{quasi-contact} if it admits a closed $2$-form $\omega$ such that $\omega^n$ is a non-zero $2n$-form all over the manifold.
\end{defi}

\noi They proved that given an almost contact manifold $\vd$ and given $\gamma\in H^2(\vd,\R)$, there exist a quasi-contact structure $\omega$ in $\vd$ such that $[\omega]=\gamma$. They also conjectured that every almost contact manifold is in fact a contact manifold.\\

\noi Now, we have:

 \begin{theo} If $k$ is even, $\ZLam$ is an almost contact manifold and also a quasi-contact manifold.
 \end{theo}
 
 \begin{proof} 
 Consider the fibration $\pi_{_{\Lam}}:\ZLam\to \vdu(\Lam)$ with fibre the circle, given by taking the quotient by the diagonal action. Since $N(\Lambda)$ is a complex manifold we have that the foliation defined by the diagonal circle action  is transversally holomorphic. Therefore,  $\ZLam$ has an atlas modeled on $\C^{n-2}\times\R$ with changes of coordinates of the charts of the form
 $$
 \left(\left(z_1,\cdots,z_{n-2}\right), t\right)  \mapsto \left(h\left(z_1,\cdots,z_{n-2},t\right), g\left(z_1,\cdots,z_{n-2},t\right)\right),
 $$ 
 \noi where $h:U\to\C^{n-2}$  and $g:U\to\R$ where $U$ is an open set in $C^{n-2}\times\R$ and, for each
 fixed $t$ the function $\left(z_1,\cdots,z_{n-2}\right)\mapsto h\left(z_1,\cdots,z_{n-2},t\right)$ is a biholomorphism onto an open set of $C^{n-2}\times\{t\}$. This means that the differential, in the given coordinates, is represented by a matrix of the form
  $$
 \left[\begin{array}{ccc|c}
  &  &  &  \\
   & A &  & \ast \\
  &  &  & \\ \hline
  0& \dots & 0 & r
  \end{array}\right],
$$
 \noi where $\ast$ denotes a column $(n-2)$-real vector and $A\in{\gl(n-2,\C)}$. 
 The set of matrices of the above type form a subgroup
 of $\gl(2n-3,\R)$. By Gram-Schmidt this group retracts onto $\su(n-2)\times\R$.\\
 
\noi This shows that $\ZLam$ is an almost contact manifold and therefore also a quasi-contact manifold by the result in \cite{MMP}.
 \end{proof}

\begin{rmk}$\;$
\begin{enumerate}
\item The conjecture in \cite{MMP} would imply our conjecture that every odd dimensional moment-angle
manifold admits a contact structure. We are working on a direct
proof of this. It would give support to the conjecture in \cite{MMP}. In section \ref{S6} we will construct contact structures in some manifolds related to moment-angle manifolds.
\item The construction in \cite{MMP} is based on an open book decomposition of the manifold. We have still not related this open book structure with the one we have in the case of a moment-angle manifold.
\end{enumerate}
\end{rmk} 

\subsection {PS-Overtwisted Manifolds.}

\noi The definition of \emph{overtwistedness} in higher dimensions was given by K. Niederkr\"uger based in the existence of a \emph{plastikstufe} (see \cite{Klaus}).

\begin{defi}
 Let $(\vd,\alpha)$ be a cooriented $(2n+1)$-dimensional contact manifold, and let $S$ be a 
closed $(n-2)$-dimensional manifold. A \emph{plastikstufe} $\ps(S)$ with singular set $S$ in $\vd$ is an embedding of the $n$-dimensional manifold 
$$
\iota: \disc^2\times\s^1 \to \vd,
$$ 
that carries a (singular) Legendrian foliation given by the $1$-form $\beta:=\iota^\ast \alpha$ sa\-tis\-fying: 
\begin{itemize}
\item [a)]  The boundary $\partial \ps(S)$ of the plastikstufe is the only closed leaf.
 
\item [b)] There is an elliptic singular set at ${0}\times S$. 

\item [c)] The rest of the plastikstufe is foliated by an $\s^1$-family of stripes, each one diffeomorphic to $(0,1)\times S$, which are spanned between the singular set on one end and approach $\partial\ps(S)$ on the other side asymptotically.
\end{itemize} 
\end{defi}

 \noi A contact manifold is called \emph{$\ps$-overtwisted} if it admits the embedding of a plastikstufe. 

\begin{obs}
 In dimension three, the definition of $\ps$-overtwisted is identical to the standard definition of overtwistedness.
\end{obs}

\begin{center}
\begin{figure}[h]
\centerline{\includegraphics[height=3cm]{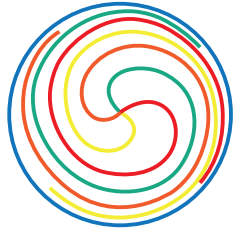}}
\caption{An overtwisted disk.}
\end{figure}
\end{center}   

\noi Following the ideas of F. Presas in the article \cite{Presas}, K. Niederk\"uger and O. Van Koert proved the next interesting results (see \cite{KO}):\\

\noi Let $\s^{2n-1}$ be the unit sphere in $\C^n$ with coordinates $Z=\left(z_1,\dots,z_n\right)\in \C^n$ and let $f$ be the polinomial 
$$
f:\C^n\to\C\quad\text{defined by}\quad \left(z_1,\dots,z_n\right)\mapsto z^2_{_1}+\dots +z^2_{_n}.
$$
The $1$-form 
$$
\alpha_-:=i\sum_{j=1}^n\left(z_jd\bar{z}_j-\bar{z}_jdz_j\right)-i\left(fd\bar{f}-\bar{f}df\right)  
$$ 
defines a contact structure on $\s^{2n-1}$ (see \cite[Proposition 7]{KO}).

\begin{theo}\cite[Corollary 4]{KO}
Every sphere $\s^{2n+1}$ with $n\geqslant 1$ supports a PS-overtwisted contact structure. More precisely, on $\s^{2n+1}$, with $n\geqslant 1$, exists a contact structure which admits the embedding of a plastikstufe PS$(\T^{n-1})$ (with $\T^0:=\{p\}$ and $\T^1:=\s^1$).
\end{theo}

\noi As a consequence of Theorem \ref{Bour}, we have the next corollary:

\begin{coro}
For $n>3$,  moment-angle manifolds such as $\ZLam=\s^{2n-5}\times\T^{2m}$ are $\ps$-overtwisted contact manifolds. 
\end{coro}

\section{Higher Dimensional Contact Manifolds.}\label{S6}

\noi Theorem 2 of \cite{SLM} was extended in \cite{GL} to the case where the manifold is given by two quadratic forms which are not necessarily simultaneously diagonalizable. This includes the manifolds we construct now:\\

\noi Let $n> 3$ and $s\geqslant 1$ be two integers and $\Lam=\left(\lambda_1,\dots,\lambda_n\right)$ be an admissible configuration with $\lambda_1\in \C$. Now consider the manifold $\ZLamss$ define by the following equations:
 
\begin{eqnarray}\label{E1}
F_s(X)&:=&\sum_{r=1}^s w^2_{_r}+\sum_{j=1}^n\lambda_j|z_j|^2=0,
\\
\label{E2}
\rho_s(X)&:=&\sum_{r=1}^s|w_r|^2+\sum_{j=1}^n|z_j|^2=1,
\end{eqnarray}
\noi where $X=\left(w_1,\dots,w_s,z_1,\dots,z_n\right)$. $\ZLamss$ has real dimension $2n+2s-3>5$.\\

\noi The topology of the manifolds $\ZLamss$ is related to the topology
of the associated moment-angle manifolds $\ZLam$ as shown by the following theorem. \\

\begin{theo}[\cite{GL}]
The manifold $\ZLamss$ is diffeomorphic to:
$$\underset{j=1}{\overset{2\ell+1}{\sharp}}\left(\s^{2d_j+s-1}\times\;\s^{2n-2d_j+s-2}\right),$$
where $d_j=n_j+\dots+n_{j+\ell-1}$.
\end{theo}

\noi We will show that for every set of $n+s$ positive numbers $a_1,\dots,a_s,b_1,\dots,b_n$ the $1$-form 
$$
\alpha:= i\left[\sum_{r=1}^s\left[a_r\left(w_rd\bar{w}_r-\bar{w}_rdw_r\right)\right]+\sum_{j=1}^n\left[b_j\left(z_j\bar{z}_j-\bar{z}_jdz_j\right)\right]\right]
 $$
can be deformed into a contact form on the manifold $\ZLamss$, by an arbitrarily small $\C^\infty$ perturbation. 

\subsection{Case $s=1$.}

\begin{lema}
For all $X=\left(w_1,z_1,\ldots,z_n\right)\in\ZLamsu$ the $1$-form $\alpha$ restricted to the tangent space $\tange_{_{X}}\left(\ZLamsu\right)$ is a nontrivial form.
\end{lema}

 \noi We will only consider the case where $a_1=1$ and $b_1=1,\,\,\,j=1,\dots, n$ since the other cases are completely analogous. 
 
\begin{proof}
 For $X=\left(w_1,z_1,\ldots,z_n\right)\in\ZLamsu$, the linear function 
$$
\alpha_{_X}:\tange_{_{X}}\left(\ZLamsu\right)\to\R
$$ 
\noi is trivial if and only if there exist $T\in\C$ and $\mu\in\R$ such that: 

\begin{eqnarray}\label{5}
\nonumber i\left[\left(w_1d\bar{w}_1-\bar{w}_1dw_1\right)+\sum_{j=1}^n \left(z_jd\bar{z}_j-\bar{z}_jdz_j\right)\right]\\
\nonumber =T\left[2w_1dw_1+ \sum_{j=1}^n\left[\lambda_j\left(\bar{z}_jdz_j+z_jd\bar{z}_j\right)\right]\right]\\
\nonumber+\overline{T}\left[2\bar{w}_1d\bar{w}_1+\sum_{j=1}^n\left[\bar{\lambda}_j\left(z_jd\bar{z}_j+\bar{z}_jdz_j\right)\right]\right]\\
+\mu\left[\left(w_1d\bar{w}_1+\bar{w}_1dw_1\right)+\sum_{j=1}^n \left(z_jd\bar{z}_j+\bar{z}_jdz_j\right)\right]. 
\end{eqnarray}

\noi Comparing coefficients in equation \eqref{5} we conclude that 
the unique point satisfying the last equations is the origin $(0, 0,\dots,0)$. However,  the origin is not in $\ZLamsu$. 
We conclude that $\alpha$ is nontrivial. 
\end{proof}

\noi We will denote by $\xi_\alpha(X)$ the kernel of $\alpha$ at the point $X\in\ZLamsu$.\\

\noi The only vectors $v\in\tange\left(\ZLamsu\right)$ where 
$\iota_vd{\alpha}=0$ are of the form: 
$$
v(T,\mu):=-i \Big(2\overline{T}\bar{w}_1+\mu w_1, (2\Re(T\lambda_1)+\mu)z_1,\ldots, (2\Re(T\lambda_n)+\mu)z_n\Big),
$$

\noi where $T\in\C$ and $t\in\R$.\\

\noi If $v(T,\mu)$ is tangent to the sphere $\s^{2n-1}$ and $w_0\neq0$ then it follows that $T$ must be of the form 
$$T=t\bar{w}^2_{_1}$$ 
with $t\in\R$.\\

\noi The condition $dF_1(v(T,\mu))=0$, when $w_1\neq0$ and $v(T,\mu)\in\s^{2n+1}$ implies:

$$
2tw^2_{_1}|w_1|^2+\mu w^2_{_1}=0,
$$

\noi which implies since $w_1\neq0$ that
$$
\mu=-2t|w_1|^2$$
\noi if $w_1=x+iy$ we have:
$$
\mu=-2t\left(x^2+y^2\right).
$$

\noi Therefore, if $\mu=0$ implies $t=0$ since, by hypothesis, what is inside the  parentheses is  positive.\\

\noi Vectors of the form
$$
v(T,\mu)=-i \Big(2\overline{T}{\bar{w}_1}+\mu w_1, (2\Re(T\lambda_1)+\mu)z_1,\ldots, (2\Re(T\lambda_n)+\mu)z_n\Big),
$$
\noi which are in $\ker\alpha$,  imply that $\mu=0$. Hence, if $w_1\neq0$ we must have that $t=0$ also.\\

\noi  For $w_1=0$ the vectors of the form $v(T,\mu)$ with $T\in\C$ and $\mu\in\R$ are in $\ker (d\alpha)$. Hence $\ker (d\alpha)$ has real dimension three.\\

\noi On the other hand,  on the set of points with $w_1=0$, the vectors of the form:
$$
-i \big(0,2\Re(T\lambda_1)z_1,\ldots, 2\Re(T\lambda_n)z_n \big), \,\,\, T\in\C
$$
\noi are all in $\ker(\alpha)\cap\ker(d\alpha)$. \\

\noi Hence $\ker(\alpha)\cap\ker (d\alpha)$ is a two dimensional vector space. We will denote this 2-dimensional vector space at the point
$X=(0,z_1,\ldots,z_n)$ by $\Pi(X)$.\\

\noi Therefore, in the set of points of $\ZLamsu$ such that $w_1\neq0$,  the form $\alpha$ is a contact form.\\

 \noi For a generic set of admissible configurations $\Lam=\left(\lambda_1,\dots, \lambda_n\right)$  the intersection of $\ZLamsu$ with the complex hyperplane with equation $w_1=0$ is a real codimension-two submanifold $W$ of $\ZLamsu$ (i.e. a submanifold of $\C^{n+1}$ of dimension 2n-3)  and $\Pi(X)\subset\tange_{_X}(W)$.

\begin{rmk}
 $W$ is essentially the moment-angle manifold $\ZLamu\subset\C^n$ where $\Lam'=\left(\lambda_1,\dots,\lambda_n\right)$.
 \end{rmk}

\subsection{Case $s>1$.}

 \begin{lema}
 For all $X=\left(w_1,\dots,w_s,z_1,\dots,z_n\right)\in\ZLamss$ the $1$-form $\alpha$ restricted to the tangent space $\tange_{_X}\left(\ZLamss\right)$ is a non trivial form.
\end{lema}

\noi As in the case $s=1$, we will consider  $a_r=b_j=1$ for all $r\in (1,\dots,s)$ and $j\in (1,\dots,n)$ since the other cases are completely analogous.

\begin{proof}
For $X=\left(w_1,\dots,w_s,z_1,\dots,z_n\right)\in\ZLamss$, the linear function
$$
\alpha:\tange_{_X}\left(\ZLamss\right)\to\R$$
 is trivial if and only if there exist $T\in \C$ and $\mu\in\R$ such that 

\begin{eqnarray}
\nonumber i\left[\sum_{r=1}^s\left(w_rd\bar{w}_r-\bar{w}_rdw_r\right)+\sum_{j=1}^n\left(z_jd\bar{z}_j-\bar{z}_jdz_j\right)\right]\\
\nonumber =T\left[\sum_{r=1}^s 2w_rdw_r+\sum_{j=1}^n\left[\lambda_j\left(\bar{z}_jdz_j+z_j\bar{z}_j\right)\right]\right] \\
\nonumber +\overline{T}\left[\sum_{r=1}^s 2\bar{w}_rd\bar{w}_r+\sum_{j=1}^n\left[\lambda_j\left(z_jd\bar{z}_j-\bar{z}_jdz_j\right)\right]\right] \\
+\mu \left[\sum_{r=1}^s \left(w_rd\bar{w}_r+\bar{w}_rdw_r\right)+\sum_{j=1}^n\left(z_jd\bar{z}_j+\bar{z}_jdz_j\right)\right].
\end{eqnarray}

\noi  Comparing coefficients we have that the points satisfying the last equations are of the form $\left(w_1,\dots,w_s,0\dots,0\right)$ where at least two $w_r$, $w_{r'}$ not zero, $r\not=r'$, $r,r'\in\{1,\dots,s\}$.  In  this cases, the equations $F_s(X)=0$ and $\rho_s(X)=1$ defines a Brieskorn manifold and $\alpha$ is a contact form on that manifolds \cite{LM}, in particular is non trivial on it.\\

\noi We conclude that $\alpha$ is non trivial on $\ZLamss$.
\end{proof}

\noi We will be denote by $\xi_\alpha(X)$ the kernel of $\alpha$ at the point $X\in\ZLamss$.\\

\noi  The unique  vectors $v\in\tange\left(\ZLamss\right)$ where $\iota_vd\alpha=0$ are of the form:

$$
v(T,\mu):=
$$
$$
-i\left(2\overline{T}\bar{w}_1+\mu w_1,\dots,2\overline{T}\bar{w}_s+\mu w_s,(2\Re(T\lambda_1)+\mu)z_1,\dots,(2\Re(T\lambda_n)+\mu)z_n\right),
$$
\noi where $T\in\C$ and $\mu\in\R$.\\

\noi If $v(T,\mu)$ is tangent to the sphere $\s^{2n+2s-1}$ and $w_r\not= 0$ for all $r\in (1,\dots,s)$  then it follows that $T$ must be of the form 
$$T=t\sum_{r=1}^s\bar{w}^2_{_r},\;\;\; t\in\R.$$

\noi The condition $dF_s(v)=0$ when $w_r\not= 0$ for all $r$ and $v(T,\mu)\in\s^{2n+2s-1}$ implies:

$$
2t\sum_{r=1}^s\left(\sum_{r=1}^s w^2_{_r}\right)|w_r|^2+\mu\sum_{r=1}^sw^2_{_r}=0,
$$
which implies, since $w_r\not=0$, that
$$
\mu=-2t\sum_{r=1}^s|w_r|^2.
$$

\noi if $w_r=x_r+iy_r$ we have:
$$
\mu=-2t\sum_{r=1}^s\left(x^2_r+y^2_r\right).
$$
\noi Therefore, if $\mu=0$ implies $t=0$ since, by hypothesis, $\sum_{r=1}^s\left(x^2_r+y^2_r\right)$ is positive.\\

\noi Vectors of the form $v(T,\mu)$ which are in $\ker(\alpha)$ imply that $\mu=0$. Hence, if $w_r\not= 0$ we must have that $t=0$ also.\\

\noi When $w_r=0$ for all $r\in\{1,\dots,s\}$, the vectors of the form $v(T,\mu)$ with $T\in\C$ and $\mu\in\R$ are in $\ker (d\alpha)$. Hence $\ker (d\alpha)$ has real dimension three.\\

\noi On the other hand, on the set of points with $w_r=0$ for all $r\in\{1,\dots,s\}$, the vectors of the form
$$
-i\big(0,\dots,0,2\Re(T\lambda_1)z_1,\dots,2\Re(T\lambda_n)z_n\big), \;\;\; T\in\C
$$
\noi are all in $\ker(\alpha)\cap\ker (d\alpha)$. \\

\noi Hence $\ker (\alpha)\cap\ker (d\alpha)$ is a two dimemsional space. We will denote this $2$-dimensional vector space at the point $X=(0,\dots,0,z_1,\dots,z_n)$ by $\Pi(X)$.\\

\noi Therefore, in the set of points of $\ZLamss$ such that $w_r\not=0$ for all $r\in\{1,\dots,s\}$, the form $\alpha$ is a contact form.\\

\noi For a generic set of admissible configurations $\Lam=(\lambda_1,\dots,\lambda_n)$ the intersection of $\ZLamss$ with the complex hyperplanes  $\left\{w_1=0\right\},\dots,\left\{w_s=0\right\}$ is a submanifold $W$ of real codimension $2s$ of $\ZLamss$.

\subsection{Conductive Confoliations.}

\noi We introduce some definitions following  the ideas of S. J. Altschuler and L. F. Wu in \cite{AltWu}:

\begin{defi}
The space of \emph{conductive confoliations}, $Con\left(\ZLamss\right)$, is defined to be the subset of $\alpha\in \Lam^1\left(\ZLamss\right)$ such that 
\begin{itemize}
\item $\alpha$ is a positive confoliation: $\ast\left(\alpha\wedge (d\alpha)^{n+s-2}\right)\geqslant 0$, where $\ast$ denotes de Hodge operator;

\item every point $p\in \ZLamss$ is accessible from a contact point $q\in\ZLamss$ of $\alpha$: there is a smooth path $\gamma:[0,1]\to \ZLamss$ from $p$ to $q$ with $\gamma'(s)$ in the orthogonal complement of $\ker\left(\ast\left(\alpha\wedge (d\alpha)^{n+s-3}\right)\right)$ for  all $s$.
\end{itemize} 
\end{defi}

\begin{theo}[Theorem 2.8, \cite{AltWu}]\label{AW}
If $\alpha\in Con\left(\ZLamss\right)$ then $\alpha$ is $C^\infty$ close to a contact form.
\end{theo}

\noi From the above and the fact that $W$ does not separate $\ZLamss$, since it is of codimension two, it follows:

\begin{prop}
 Let $\ast$ denote the Hodge operator for a given Riemannian metric on $\ZLamss$, then for the appropriate orientation of $\ZLamss$ one has that
$$
\ast\left(\alpha\wedge (d\alpha)^{n+s-2}\right)(X)>0,
$$
for $X\notin{W}$ and  if $X\in W$
$$
\ast\left(\alpha\wedge (d\alpha)^{n+s-2}\right)(X)=0.
$$
Therefore $\alpha$ is a positive confoliation on $\ZLamss$.
\end{prop}

\noi We must show that  every point is accessible from a contact point of $\alpha$.
 
\begin{lema} Let $X\in{W}$. Then, there exists a smooth parametrized curve
$\gamma:[0,1]\to\ZLamss$ such that $\gamma(0)=X$, $\gamma(s)\notin{W}$ if $s\in(0,1]$ and $\gamma'(s)$ is a nonzero vector such that $\gamma'(s)\in\xi_\alpha(\gamma(s))$.
\end{lema}

\begin{proof}
\noi  Let us fix a Riemannian metric $g$. For $P\in W$ let $\vt(P)$ denote the $2$-dimensional subespace of $\tange\left(\ZLamss\right)$ which is orthogonal to $\tange_{_P}W$ at $P$.\\

\noi Let us first show  that there exists an open neighborhood $X\in\U\subset W$ of $X\in W$ and a smooth and non-vanishing vector field $\cv:\U\to \tange_{_X}\left(\ZLamss\right)$ defined on $\U$ such that

\begin{enumerate}
\item $\cv(X)\in\xi_\alpha(X)\cap \vt(X)$,
\item $\cv(P)\in \xi_{\alpha}(P)$ for all $P\in\U$, 
\item $\cv(P)\notin\tange_{_P}(W)$ for all $P\in\U$. 
\end{enumerate}

\noi Indeed, Let $\ls(X)=\vt(X)\cap\xi_\alpha(X)$. Then $\ls(X)$ has dimension two if $\vt(X)\subset\xi_\alpha(X)$ or $\ls(X)$ has dimension one if $\vt(X)$ is transverse to $\xi_\alpha(X)$.\\

\noi  Let $v_{_X}\in \ls(X)$ be a non zero vector. Extend this vector anchored at $X$ to a smooth vector field $\tilde{\cv}:\K\to\tange\left(\ZLamss\right)$ defined in a neighborhood $X\in\K\subset W$ of $X\in W$. Let $\pi_{_P}:\tange_{_P}\left(\ZLamss\right)\to\xi_\alpha(P)$ be the orthogonal projection. Consider the vector field defined on $\K$ by $\cv_1(P)=\pi_{_P}(\tilde{\cv}(P))$. Then $\cv_1$ is a smooth vector field and by continuity $\cv_1$ satisfies all the required properties in a possible smaller neighborhood $\U$.\\

\noi  To finish the proof of the lemma we have that by standard extension theorems (partition of unity) there exists an extension of $\cv_1$ to a nonsingular vector field $\cv_2:\V\to\tange_{_X}\left(\ZLamss\right)$ defined on an open neighborhood $X\in\V\subset \ZLamss$ of $X$. The vector field defined by $\cv(P)=\pi_{_P}(\cv_2(P))$ has the property that $\cv(P)\in\xi_\alpha(P)$.\\

\noi By multiplying the vector field $\cv$ by a positive constant $c>0$, if necessary, we can assume that all the solutions of the differential equation defined by the vector field $c\cv$ on $\V$ are defined in the interval $(-2,2)$.\\

\noi If $\gamma:[0,1]\to\U$ is the solution of the differential equation determined by $c\cv$ and satisfying the initial condition $\gamma(0)=X$ then this parametrized curve satisfies the requirements of the theorem \ref{AW} if $c$ is sufficiently small.
\end{proof}

\noi The proposition implies that every point of $W$ can be joined, by a Legendrian path of finite length, to a point where the form  $\ast\left(\alpha\wedge (d\alpha)^{n+s-2}\right)(X)$ is positive.\\ 

\noi Since $\alpha$ is a contact form on $\ZLamss-W$ we have that any two points of $\ZLamss-W$ can be joined by a Legendrian curve. Therefore every point of $\ZLamss$ is \emph{accessible} and  $\alpha$ defines a \emph{conductive confoliation} in the sense of J. S. Altschuler and L. F.  Wu (\cite{AltWu}).\\

\noi These forms are also called \emph{transitive confoliations} by Y. Eliashberg and W. P. Thurston \cite{ET}, since we can connect any point of the manifold  to a point where the form $\alpha$ is contact by a Legendrian path of finite length.\\

\noi Therefore applying Theorem \ref{AW} we can deform $\alpha$ to a contact form $\alpha'$.  Furthermore $\alpha'$ can be chosen arbitrarily close to $\alpha$ in the $C^\infty$ topology.  We have the theorem:

\begin{theo} For a generic set of admissible configurations $\Lam=\left( \lambda_1,\ldots, \lambda_n\right)$, the manifold $\ZLamss$ is a contact manifold. \end{theo}

\noi In other words:

\begin{theo}
Let $n>3$, $s\in\{1,\dots,n\}$ and let $\Lam$ be an admissible configuration. The manifolds 
$$\underset{j=1}{\overset{2\ell+1}{\sharp}}\left(\s^{2d_j+s-1}\times\s^{2n-2d_j+s-2}\right),$$
where $d_j=n_j+\dots+n_{j+\ell-1}$ admit contact structures.
\end{theo}

\begin{rmk}\label{rmk4}
 It was shown by C. Meckert \cite{Meckert} that the connected sum of contact manifolds of the same dimension
is a contact manifold. It was pointed to us by Dishant Pancholi that this implies the manifolds $\ZLamss$ have a contact structure.\\

\noi Indeed, the manifolds $\ZLamss$ are connected sums of products of the form $\s^n\times\s^m$ with $n$ even and $m$ odd, and $n,m>2$. Whitout loss of generality, we suppose that $m>n$ (the other case is analogous) then $\s^m$ is an open book with binding $\s^{m-2}$ and page $\R^{m-1}$. Hence $\s^n\times\s^m$ is an open book with binding $\s^{m-2}\times\s^n$ and page $\R^{m-1}\times\s^{n}$. The page  $\R^{m-1}\times\s^n$ is paralellizable since it embeds as an open subset of $\R^{m+n-1}$, therefore, since $m+n-1$ is even it has an almost complex structure. Furthermore, by hypothesis, $2n\leq{n+m}$ hence by a theorem of Y. Eliashberg \cite{Elias} the page is Stein and is the interior of an compact manifold with contact boundary $\s^{m-2}\times\s^n$. Hence by  a theorem of E. Giroux \cite{Giroux} $\s^n\times\s^m$ is a contact manifold. However our construction is in some sense explicit  since it is the instantaneous  diffusion through the heat flow of an explicit 1-form which is a positive confoliation.\\
\end{rmk}

\begin{rmk}
\noi Another interesting fact is that the manifolds $\ZLamss$ also have an open book decomposition. However for  these open book decompositions there does not exist a contact form which is supported in the open book decomposition  like in Giroux's theorem 
because the pages are not Weinstein manifolds 
(i.e manifolds of dimension $2n$ with a Morse function with indices of critical points lesser or equal to $n$).
\end{rmk}

\section{Topology of Intersections of quadrics.}

In this last section we recall some old results about the topology of intersections of quadrics $Z(\Lam)$ and the ideas behind their proofs. We use these to proof a new result about the topology of the manifolds with boundary $Z_+(\Lam)$ 
from which Theorem  \ref{pageC} of section 2 follows.

\subsection{Homology of Intersections of quadrics.}

We recall here the results of \cite{SLM}, whose proofs are equally valid for  any intersection of quadrics and not only  for $k=2$:\\

\noi  Let $Z=Z(\Lam)\subset \R^n$ as before, $\p$ its associated polytope  and $F_1,\dots,F_n$ its facets obtained by intersecting $\p$ with the coordinate hyperplanes $x_i=0$ (some of which might be empty).\\

\noi Let $g_i$ be the reflection on the $i$-th coordinate hyperplane and for $J\subset\{1,\dots,n\}$ let $g_J$ be the composition of the $g_i$ with $i\in J$.\\

\noi Let also $F_J$ be the face obtained by intersecting the facets $F_i$ for $i\in  J$.\\

\noi The polytope $\p$, all its faces, and all their combined reflections on the coordinate hyperplanes form a cell decomposition of $Z$. Then the elements $g_J(F_L)$ are generators of the chain groups $C_\ast(Z)$, where to avoid repetitions one has to ask $J\cap L=\emptyset$ (since for any $i$ the intersection $g_i$ acts trivially on $F_i$).\\

\noi A more useful basis is given as follows: let $h_i=1-g_i$ and $h_J$ be  the product of the $h_i$ with $i\in J$. The elements $h_J(F_L)$ with $J\cap L=\emptyset$ are also a basis, with the advantage that $h_JC_\ast(Z)$ is a chain subcomplex for every $J$ and, since $h_i$ annihilates $F_i$ and all its subfaces, it can be identified with the chain complex $C_\ast(\p,\p_J)$, where $\p_J$ is the union of all the facets $F_i$ with $i\in J$. It follows that
$$
H_\ast(Z)\approx\oplus_JH_\ast(\p,\p_J).
$$

\noi For the manifold $Z_+$ we have the same elements, but we cannot reflect them
in the subspace $x_1=0$. This means we miss the classes $h_J(F_L)$ where $1\in J$ and we get\footnote{The retraction $Z\to Z_{_+}$ induces an epimorphism in homology and fundamental group.}
$$H_\ast(Z_{_+})\approx\oplus_{1\notin J} H_\ast(\p,\p_J).$$

\noi These splittings are consistent with the ones  derived from the homotopy splitting of $\Sigma Z$ described in \cite{BBCG}. Even if it is not clear that they are the \emph{same} splitting, having two such with different geometric interpretations is most valuable. 

\subsection{Topology of $Z$ when $k=2$.}

Let us recall the results of \cite{SLM} for $k=2$. In this case, every intersection of quadrics is diffeomorphic to one of the following particular forms:\\

\noi Take $n=n_1+\dots+n_{2\ell+1}$ a partition of $n$, where we will not distinguish between a partition and any of its permutations preserving the cyclic order and we will think of the index $i$ in $n_i$ as an interger mod $2\ell+1$.  Corresponding to it we have the configuration $\Lam$ consisting of the $(2\ell+1)$-th roots of unity, the $i$-th one taken with multiplicity $n_i$. In other words, we have $\{1,\dots,n \}=J_1\cup\dots\cup J_{2\ell+1}$ in increasing order where the size of $J_i$ is $n_i$ and for every $j\in J_i$ one has $\lambda_j$ is the $i$-th of the $(2\ell+1)$-th roots of unity.

\begin{center}
\begin{figure}[h]
\centerline{\includegraphics[height=5cm]{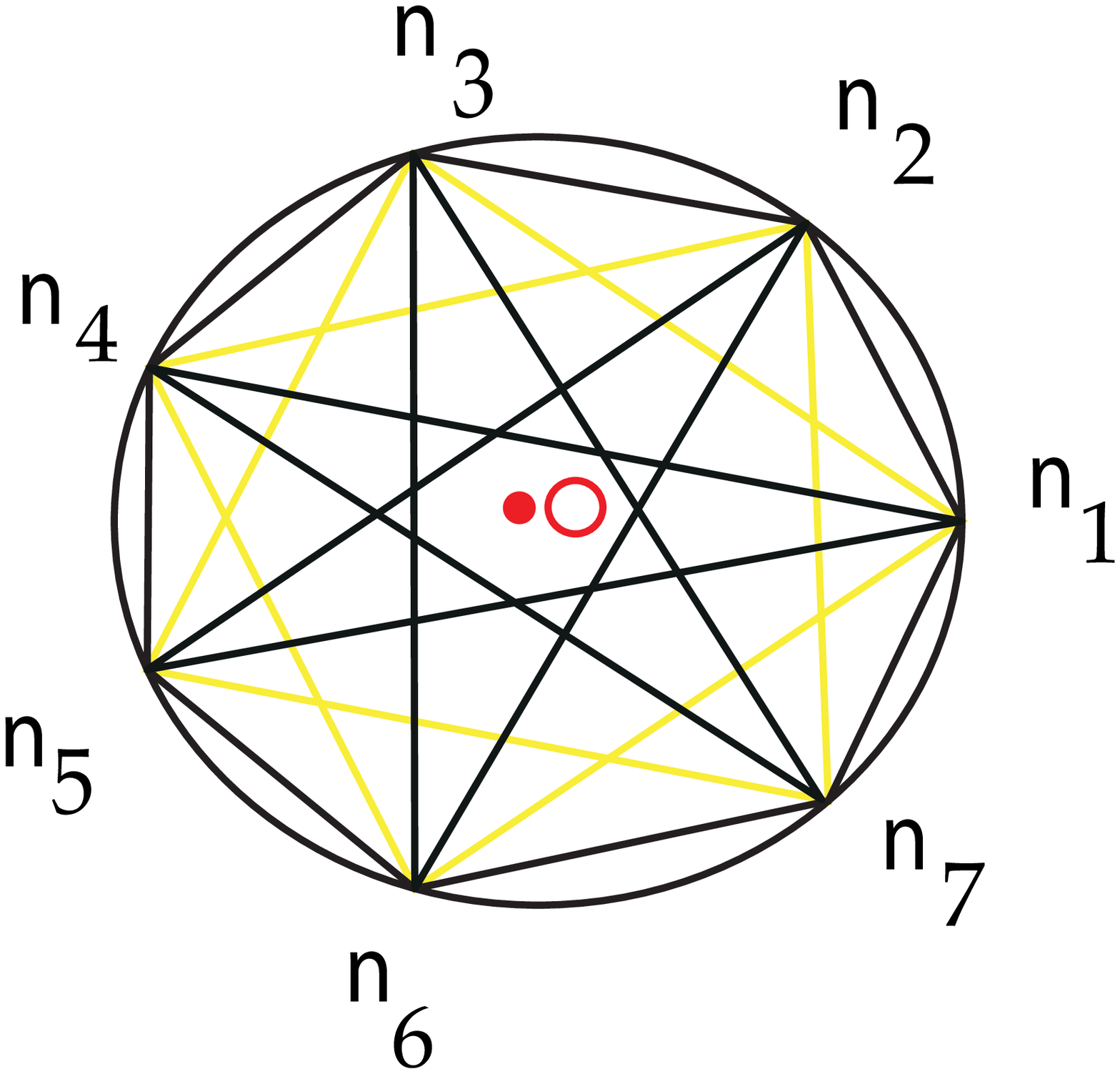}}
\end{figure}
\end{center}

\noi Any configuration can be deformed into one of these by concentrating in one point all the coefficients that is possible without breaking the weak hy\-per\-bo\-li-ci\-ty condition. We think of the $J_i$ as classes of points that can be joined this way.\\

\noi The pairs $(\p,\p_J)$ with non-trivial homology are those where $J$ consists of $\ell$ or $\ell+1$ consecutive classes, that is, those where $J$ is either one of  the $D_i=J_i\cup\dots\cup J_{i+\ell-1}$ or one of their complements $\tilde{D}_i$ in $\{1,\dots,n\}$. In those cases there is just one dimension where the homology is non-trivial and  it is infinite cyclic.

\begin{center}
\begin{figure}[h]
\centerline{\includegraphics[height=5cm]{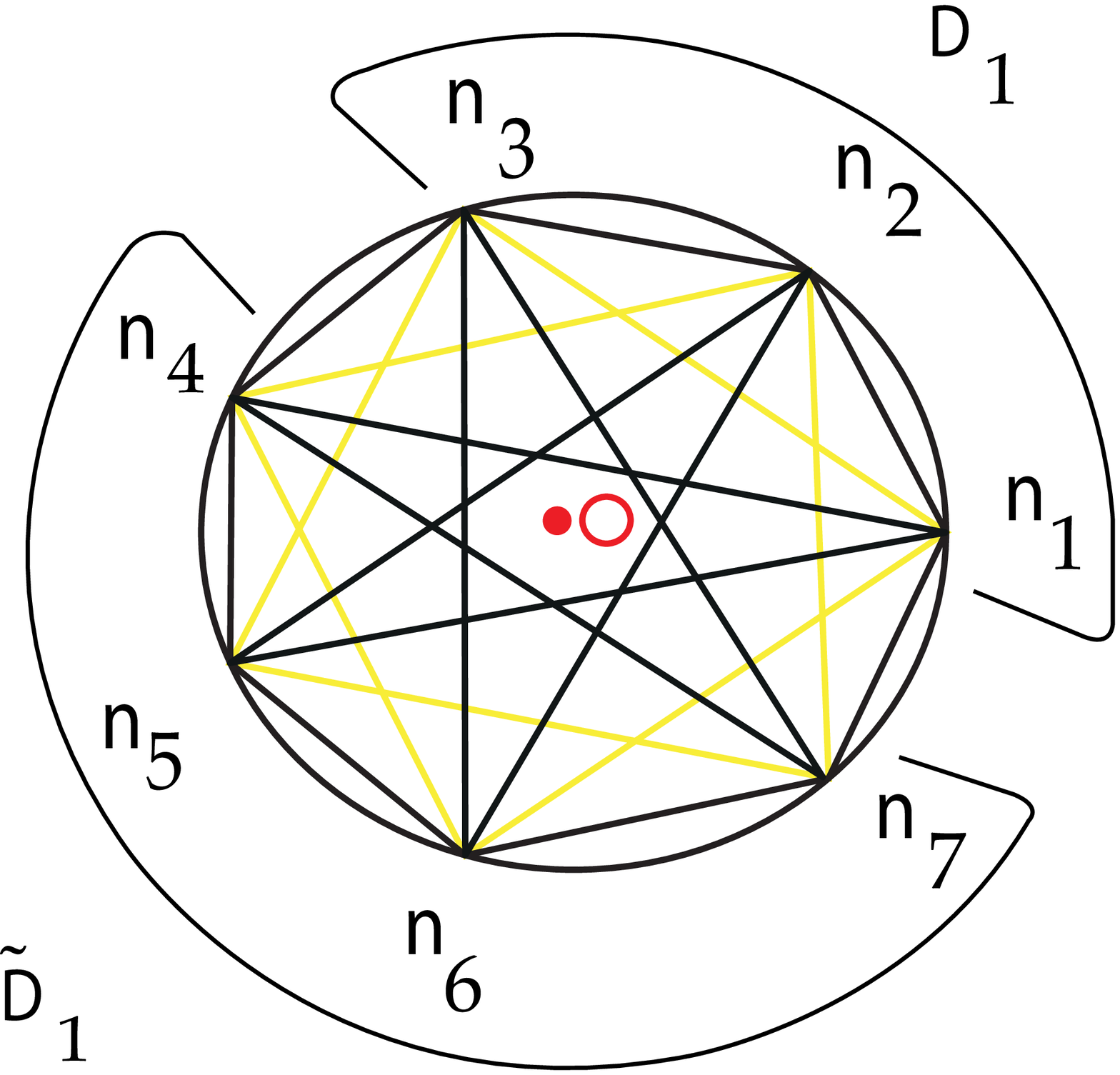}}
\end{figure}
\end{center}   

\noi  In the case of $D_i$ that homology group is in dimension $d_i-1$ where $d_i=n_i+\dots+n_{i+\ell-1}$ is the length of $D_i$. To specify a generator consider the cell $F_{L_i}$ where $L_i=\tilde{D}_i\backslash\left(\{j_{i-1}\}\cup\{j_{i+\ell}\}\right)$ and $j_{i-1}\in J_{i-1}$, $j_{i+\ell}\in J_{i+\ell}$ are any two indices in the extreme classes of $\tilde{D}_i$ (in other words, those contiguous to $D_i$).\\
 
 \noi $F_{L_i}$ is non empty of dimension $d_i-1$. It is not in $\p_{D_i}$, but its boundary is. Therefore it represents a homology class in $H_{d_i-1}(\p,\p_{D_i})$, which is actually a generator, and defines a generator $h_{D_i}F_{L_i}$ of $H_{d_i-1}(Z)$. Since $F_{L_i}$ has exactly $d_i$ facets it is a $d_i-1$-simplex so when reflected in all the coordinate subspaces containing those facets we obtain a sphere, which clearly represents $h_{D_i}F_{L_i}\in H_{d_i-1}(Z)$.\\
 
 \noi The class corresponding to $\tilde{D}_i$ is in dimension $n-d_i-2$ and is the Poincar\'e dual of the the one corresponding to $D_i$. One gets again easily a representative which in this case is not a sphere, but which can be turned into one with a good choice and a surgery if $\ell>1$. For our purpose we do not need to say more about this case.\\
 
 \noi The net result is that, if $\ell>1$, all the homology below the top dimension can be represented by embedded spheres\footnote{This also follows from \cite {LoGli}.} with trivial normal bundle which can be made disjoint inside $Z_+'$ and (since the inclusion $Z\subset Z_{_+}'$  induces an epimorphism in homology due to the homotopy equivalence $Z_+\subset Z_+'$) we represent all the classes in $H_\ast(Z_{_+}')$ by spheres, the $h$-cobordism theorem shows that this manifold is a connected sum along the boundary of manifolds of the form $\s^p\times\disc^{n-3-p}$. Then its boundary $Z$ is a connected sum of spheres products. Knowing its homology we can tell the dimensions of those spheres:\\

\noi \emph{If $\ell>1$ and $Z$ is simply connected of dimension at least $5$, then\footnote{The result has recently been proved in \cite{GL} without the dimension and conectivity hypotheses}:}
$$
Z=\sum_{j=1}^{2\ell+1}\left(\s^{d_i-1}\times\s^{n-d_i-2}\right).
$$
When $\ell=1$ a simple computation shows that
$$
Z=\s^{2n_1-1}\times\s^{2n_2-1}\times\s^{2n_3-1}.
$$

\subsection{Topology of $Z_+$ when $k=2$.}

The topology of $Z_+'$ is implicit in the above proof, and since any $Z$  with $n_1>1$ is such a $Z'$ so we have:\\

\noi \emph{If $Z_0$ is simply connected of dimension at least $5$, and $\ell>1$, $n_1>1$ then:}
$$
Z_+=\coprod_{i=2}^{\ell+2}\left(s^{d_i-1}\times\disc^{n-d_i-2}\right)\coprod\coprod_{i=\ell+3}^1\left(\disc^{d_i-1}\times\s^{n-d_i-2}\right).
$$

\noi The case $n_1=1$ has to be considered separately. The difference in the topology of $Z_+$ between case $n>1$ and $n=1$ can be seen as follows:\\

\noi As mentioned before, in the first case the map $Z_0\to Z_+$ induces an epimorphism in homology.\\

\noi This is not the case for $n_1=1$. For example, take the case $5=1+1+1+1+1$. Here $Z_0$ consists of four copies of $\s^1$ while $Z_+$ is a torus minus four disks\footnote{In this case the polytope is a pentagon and the Euler characteristics of $Z$ and $Z_0$ follow easily by looking at its faces.} Or, equivalently, a sphere minus four disks (where all the homology comes from the boundary) with a handle attached that carries the homology not coming from the boundary.\\

\noi The main fact is that $Z_0$ is given by $2\ell-1$ classes, and has only $4\ell-2$ homology generators below the top dimension, only half of which survive in $Z_+$. While the latter manifold has $2\ell+1$ homology generators.\\

\noi To be more precise, the removal of the element $1\in I_1$ allows the opposite classes $I_{\ell+1}$ and $I_{\ell+2}$ to be joined into one without breaking the weak hyperbolicity condition.

\begin{center}
\begin{figure}[h]
\centerline{\includegraphics[height=5cm]{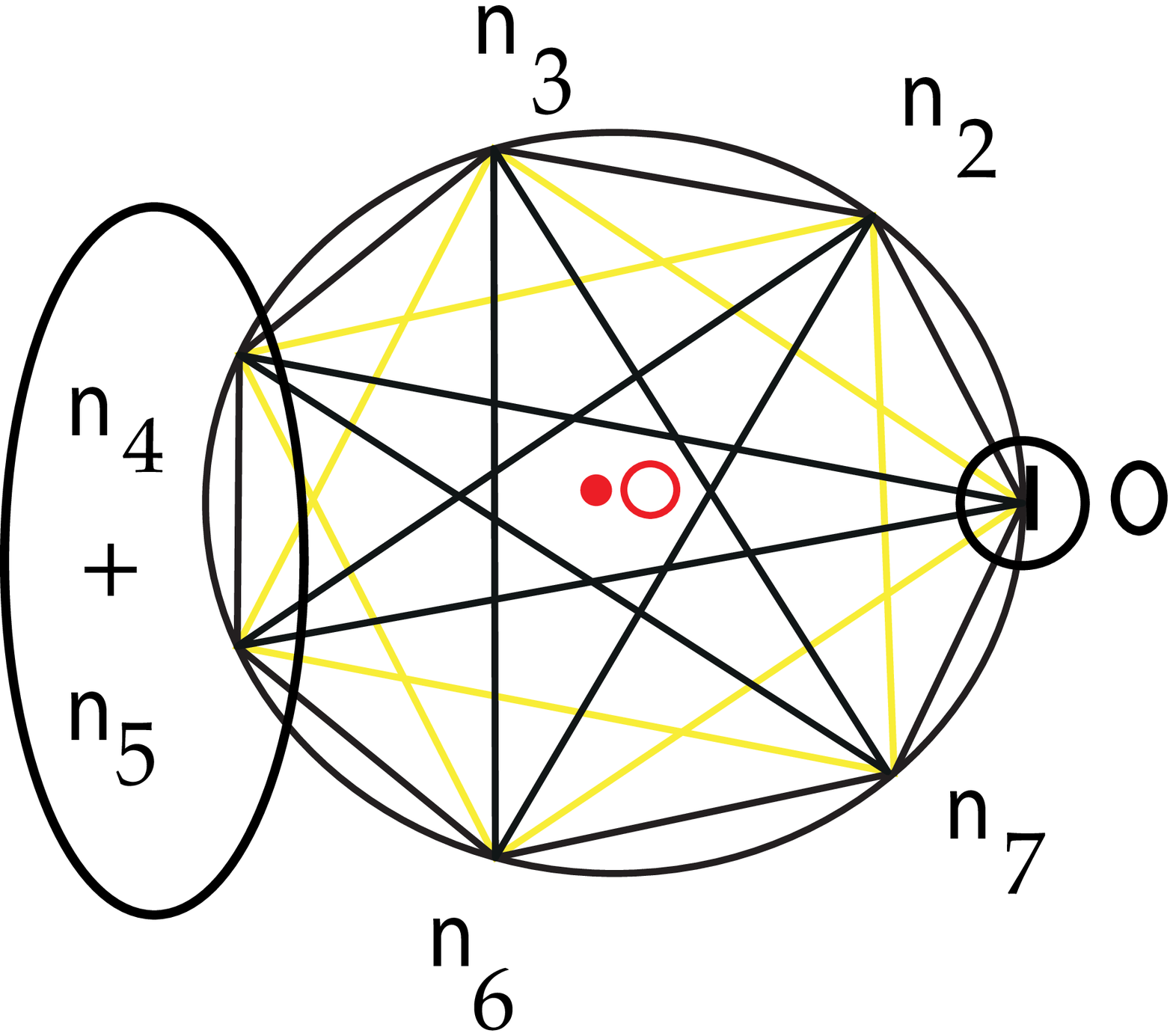}}
\end{figure}
\end{center}

\noi Therefore $Z_0$ has fewer such classes and $D_2=I_2\cup\dots\cup I_{\ell+1}$, which gives a generator of $H_\ast(Z_{_+})$, does not give anything in $H_\ast(Z_{_0})$ because there \emph{it is not a union of classes}  (it lacks the elements of $I_{\ell+2}$ to be so).\\

\noi The two classes in $H_\ast(Z_{_+})$ missing in $H_\ast(Z_{_0})$ are thus those corresponding to $J=D_2$ and $J=D_{\ell+2}$, all the others contain both $I_{\ell+1}$ and $I_{\ell+2}$ and thus live in $H_\ast(Z_{_0})$.\\

\noi As shown above, these two classes are represented by embedded spheres in $Z_+$ with trivial normal bundle built from the cells $F_{L_2}$ and $F_{L_{\ell+2}}$ by reflection, where
$$
L_2=\tilde{D}_2\backslash \left(\{1\}\cup\{j_{\ell+2}\right\}),1\in J_1,j_{d+2}\in J_{d+2}.
$$

$$
L_{\ell+2}=\tilde{D}_{\ell+2}\backslash \left(\{j_{\ell+1}\cup\{1\}\right), j_{\ell+1}\in J_{\ell+1}, 1\in J_{2\ell+2}=J_1.
$$

\noi $F_{L_2}\cap F_{L_{\ell+2}}$ is a single vertex $v$, all coordinates except $x_1$, $x_{j_{\ell+1}}$, $x_{j_{\ell+2}}$  being $0$.

\begin{center}
\begin{figure}[h]
\centerline{\includegraphics[height=5cm]{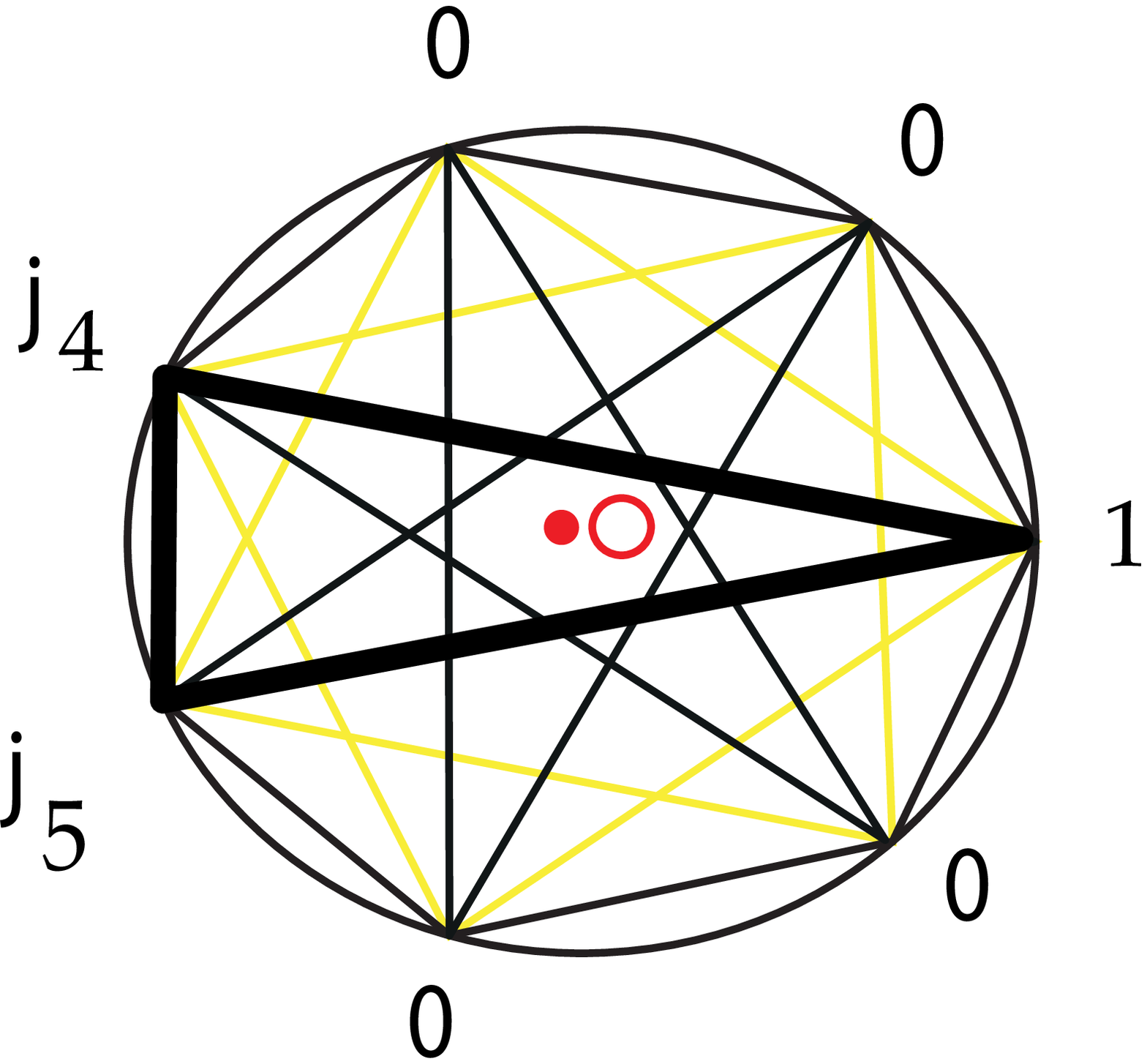}}
\end{figure}
\end{center}   
\noi The corresponding spheres are obtained by reflecting in the hyperplanes co\-rres\-pon\-ding to elements in $D_2$ and $D_{\ell+2}$, respectively. Since these sets are disjoint, the spheres also intersect in a single point.\\

\noi Now, a neighborhood of the vertex $v$ in $\p$ looks like the first orthant of $\R^{n-3}$ where the faces $F_{L_2}$ and $F_{L_{l+2}}$ correspond to complementary subspaces. When reflected in all the coordinates hyperplanes of $\R^{n-3}$, one obtains a neighborhood of $v$ in $Z_+$ where those subspaces produce neighborhoods   of the two spheres. Therefore the spheres intersect transversely in that point. (This shows again that those two classes do not come from the boundary $Z_0$: any homology class coming from the boundary can be separated from any other homology class in $Z_+$ and so has trivial homology intersection with it).\\

\noi A regular neighborhood of the union of those spheres is diffeomorphic to their product minus a disk, specifically to
$$
 \s^{d_2-1}\times\s^{d_{\ell+2}-1}\backslash \disc^{n-3}.
 $$

\noi If $\ell>2$ the rest of the classes coming from $Z_0$ can be represented again by disjoint products $\s^p\times\disc^{n-p-3}$ so finally the $h$-cobordism theorem gives\\

\noi \emph{If $Z$ is simply connected of dimension at least $6$, and $n_1=1$, $\ell>2$ then:}

$$
Z_{_+}=\coprod_{i=3}^{\ell+1}\left(\s^{d_i-1}\times\disc^{n-d_i-2}\right)\coprod\coprod_{i=\ell+3}^1\left (\disc^{d_i-1}\times\s^{n-d_i
-2}\right)
$$
$$
\coprod \left(\s^{d_2-1}\times\s^{d_{\ell+2}-1}\backslash \disc^{n-3}\right).
$$

\noi When $n_1=1$ and $\ell=2$ we have the additional complication that making $x_1=0$ we pass from the \emph{pentagonal} $Z_+$ to the \emph{triangular} $Z_0$ , which is not a connected sum but a product of three spheres and, furthermore, not all its homology below the middle dimension is spherical. All we can say from the above is that 
$$
Z_{_+}=\left(\s^{d_2-1}\times\s^{d_4-1}\backslash \disc^{n-3}\right)\;\sharp\; \tilde{\mathcal{E}}
$$ 
where the latter manifold carries all the homology coming from the boundary $Z_0$.\\

\noi We need to be more precise:\\

\noi The homology generators in $Z_+$ are those corresponding to $n_2+n_3$, $n_3+n_4$, $n_4+n_5$, $n_2+n_3+n_4$ and $n_3+n_4+n_5$. Out of these $n_2+n_3$ and $n_4 + n_5$ give us the handle which is not in $\te$, so the classes in this part are $n_3+n_4$, $n_2+n_3+n_4$ and $n_3+n_4+n_5$ all coming from its boundary, which is $Z_{_0}=\left(\s^{n_2-1}\times\s^{n_3+n_4-1}\times\s^{n_5-1}\right)$. They can all be killed by taking the union $\Sigma$ of $\te$ with $W=\left(\s^{n_2-1}\times\disc^{n_3+n_4}\times\s^{n_5-1}\right)$ along their common boundary.

\begin{center}
\begin{figure}[h]
\centerline{\includegraphics[height=5cm]{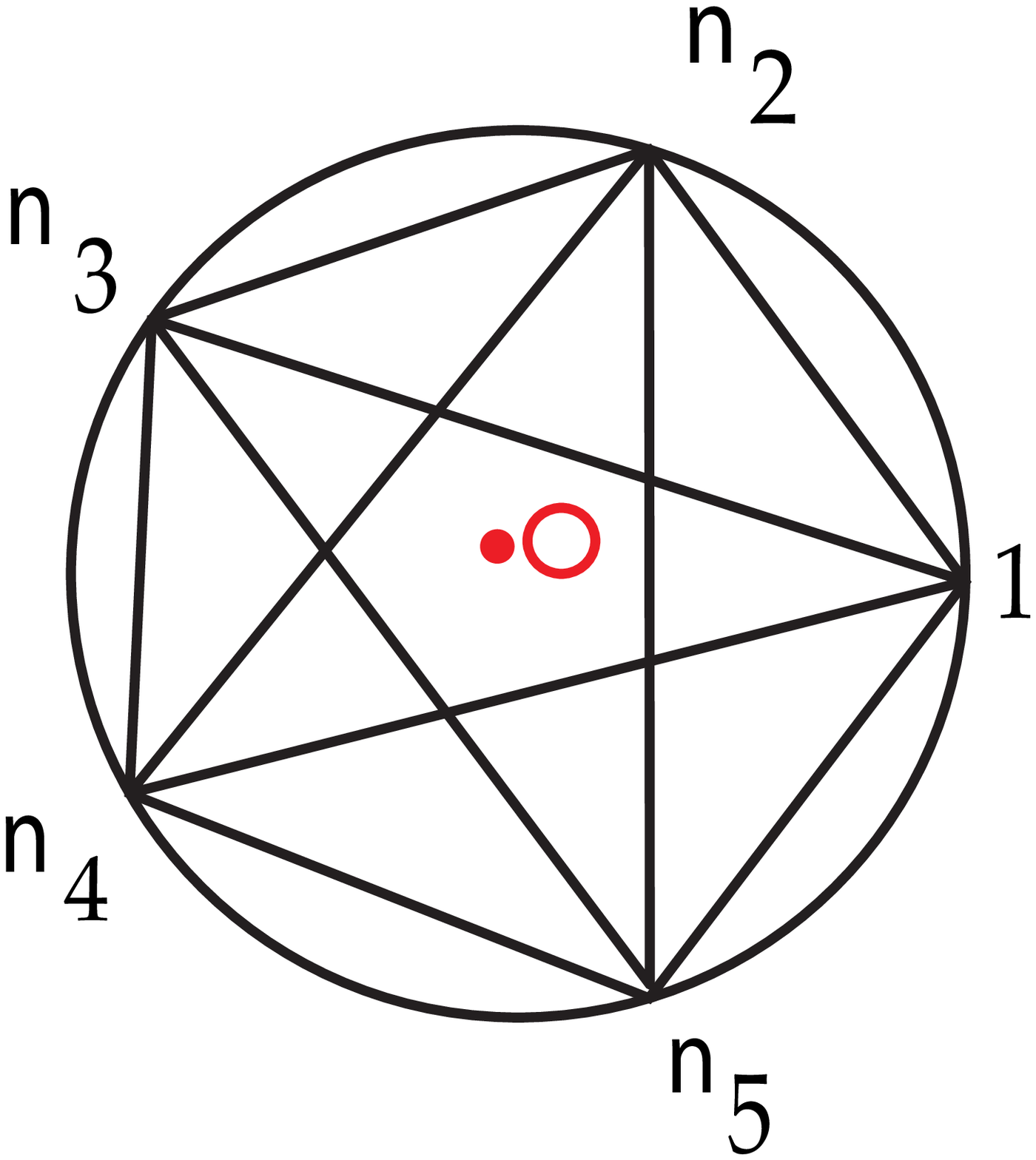}}
\end{figure}
\end{center}   

\noi $\Sigma$ is an $(n-3)$-manifold without boundary and ones sees from the  
above that the inclusions of $Z_0$ in $\te$ and $W$ induce an isomorphism below the top dimension
$$ 
H_\ast(Z_{_0})\approx H_\ast(\tilde{\mathcal{E}})\bigoplus H_\ast(W).
$$

\noi This implies (by the Mayer-Vietoris sequence) that $\Sigma$ is a homology sphere. It is simply-connected if $Z_+$ is so because this implies that 
$\te$ is so too and that $Z_0$ is connected.\\

\noi So $\Sigma$ is a homotopy sphere and $\te$ is the exterior of $W$ inside a homotopy sphere. Adding the inverse of that homotopy sphere at the interior of $\te$ does not change its diffeomorphism type but turns $\Sigma$ into a standard sphere. Therefore $\te$ is diffeomorphic to the exterior of a $\s^{n_2-1}\times\s^{n_5-1}$ in $\s^{n-3}$. It can be expected that this embedding can be assumed to be standard, but for the moment we have the following result:

\begin{theo}
Let $k=2$, and consider the manifold $Z$ corresponding to the cyclic decomposition $n=n_1+\dots+n_{2\ell+1}$ and the half manifold $Z_{_+}=Z\cap\{x_1\geqslant 0\}$. When $\ell>1$ assume $Z$ and $Z_{_0}=Z\cap\{x_1=0\}$ are simply connected and the dimension of $Z$ is at least $6$. Then $Z_+$ diffeomorphic to:
\begin{itemize}
\item [a)] If $\ell=1$, the product 
$$
\s^{n_2-1}\times\s^{n_3-1}\times\disc^{n_1-1}.
$$
\item [b)] If $\ell>1$ and $n_1>1$, the connected sum along the boundary of $2\ell+1$ manifolds:
$$
\coprod_{i=2}^{\ell+2}\left(\s^{d_i-1}\times\disc^{n-d_i-2}\right)\coprod
\coprod_{i=\ell+3}^1\left(\disc^{d_i-1}\times\s^{n-d_i-2}\right).
$$
\item [c)] If $n_1=1$ and $\ell>2$, the connected sum along the boundary of $2\ell$ manifolds:
$$
\coprod_{i=3}^{\ell+1}\left(\s^{d_i-1}\times\disc^{n-d_i-2}\right)\coprod\coprod_{i=\ell+3}^1\left(\disc^{d_i-1}\times\s^{n-d_i-2}\right)
$$
$$
\coprod\left(\s^{d_2-1}\times\s^{d_{\ell+2}-1}\backslash \disc^{n-3}\right).
$$
\item [d)] If $n_1=1$ and $\ell=2$, a connected sum along the boundary of two manifolds:
$$
\left(\s^{d_2-1}\times\s^{d_4-1}\backslash\disc^{n-3}\right)\coprod\te,
$$ 
where $\te$ is the exterior of an embedded $\s^{n_2-1}\times\s^{n_5-1}$ in $\s^{n-3}$.
\end{itemize}
\end{theo}

\noi This theorem clearly implies the following result for the topology of the page of $Z'$, which includes \ref{pageC}:
  
  \begin{theo}
  Let $k=2$, and consider the manifold $Z$ corresponding to the cyclic partition $n=n_1+\dots+n_{2\ell+1}$. When $\ell>1$ assume $Z$  and $Z_{_0}=Z\cap\{x_1=0\}$ are simply connected and the dimension of $Z$ is at least $6$. Consider the open book decomposition of $Z'$ given by theorem \ref{bookR}. Then the leaf of this decomposition is diffeomorphic to the interior of:
  \begin{itemize}
\item [a)] If $\ell=1$, the product 
$$
\s^{n_2-1}\times\s^{n_3-1}\times\disc^{n_1-1}.
$$

\item [b)] If $\ell>1$ and $n_1>1$, the connected sum along the boundary of $\ell+1$ manifolds:
$$
\coprod_{i=2}^{\ell+2}\left(\s^{d_i-1}\times\disc^{n-d_i-2}\right)\coprod\coprod_{i=\ell+3}^1\left(\disc^{d_i-1}\times\s^{n-d_i-2}\right).
$$

\item [c)] If $n_1=1$ and $\ell>2$, the connected sum along the boundary of $2\ell$ manifolds:
$$
\coprod_{i=3}^{\ell+1}\left(\s^{d_i-1}\times\disc^{n-d_i-2}\right)\coprod\coprod_{i=\ell+3}^1\left(\disc^{d_i-1}\times\s^{n-d_i-2}\right)$$
$$
\coprod\left(\s^{d_2-1}\times\s^{d_{\ell+2}-1}\backslash \disc^{n-3}\right).
$$

\item [d)] If $n_1=1$ and $\ell=2$, a connected sum along the boundary of two manifolds:
$$
\left(\s^{d_2-1}\times\s^{d_4-1}\backslash\disc^{n-3}\right)\coprod\te,
$$
 where $\te$ is the exterior of an embedded $\s^{n_2-1}\times\s^{n_5-1}$ in $\s^{n-3}$.
 \end{itemize}
  \end{theo}
  
\bigskip
\noi {\bf Acknowledgments.} We would like to thank professors Dishant Pancholi and Steven J. Altschuler for very useful comments.

\end{document}